\documentclass[english,12pt,oneside]{amsproc}
\usepackage[english]{babel}
\usepackage{a4wide}
\usepackage{amsthm}
\usepackage{graphics}
\usepackage{amsfonts, amssymb, amscd, amsmath}
\usepackage{latexsym}
\usepackage[matrix,arrow,curve]{xy}
\usepackage{mathabx}
\usepackage{color}
\usepackage{pbox}
\usepackage{tikz}
\usetikzlibrary{matrix,decorations.pathreplacing,positioning}

\DeclareMathOperator{\Hom}{Hom} 
 \DeclareMathOperator{\rk}{rk}
\DeclareMathOperator{\diag}{diag}
 
 \DeclareMathOperator{\Fl}{Fl}
\DeclareMathOperator{\Tr}{Tr} \DeclareMathOperator{\grad}{grad}

\DeclareMathOperator{\Hess}{Hess}

\newcommand{\Zo}{\mathbb{Z}}
\newcommand{\Ro}{\mathbb{R}}

\newcommand{\Co}{\mathbb{C}}

\newcommand{\Zt}{\Zo_2}

\newcommand{\RP}{\mathbb{R}P}
\newcommand{\CP}{\mathbb{C}P}

\newcommand{\twist}{\xi}
\newcommand{\hmin}{h_{\min}}
\newcommand{\hm}{h_{\max}}
\newcommand{\uu}{\mathfrak{u}}
\newcommand{\ttt}{\mathfrak{t}}

\newcounter{stmcounter}[section]

\numberwithin{equation}{section}

\theoremstyle{plain}
\newtheorem{cor}[stmcounter]{Corollary}

\newtheorem{thm}[stmcounter]{Theorem}

\newtheorem{prop}[stmcounter]{Proposition}
\newtheorem{lem}[stmcounter]{Lemma}
\newtheorem{probl}[stmcounter]{Problem}

\theoremstyle{definition}
\newtheorem{defin}[stmcounter]{Definition}

\theoremstyle{remark}
\newtheorem{ex}[stmcounter]{Example}
\newtheorem{rem}[stmcounter]{Remark}

\begin{document}

\title{Manifolds of isospectral matrices and Hessenberg varieties}

\author{Anton Ayzenberg, Victor Buchstaber}
\address{V.A. Steklov Mathematical Institute, RAS, Moscow, Russia}
\email{ayzenberga@gmail.com}
\address{V.A. Steklov Mathematical Institute, RAS, Moscow, Russia}
\email{buchstab@mi.ras.ru}

\date{\today}
\thanks{This work is supported by the Russian Science Foundation under grant 14-11-00414.}
%

\begin{abstract}
We consider the space $X_h$ of Hermitian matrices having staircase
form and the given simple spectrum. There is a natural action of a
compact torus on this space. Using generalized Toda flow, we show
that $X_h$ is a smooth manifold and its smooth type is independent
of the spectrum. Morse theory is then used to show the vanishing
of odd degree cohomology, so that $X_h$ is an equivariantly formal
manifold. The equivariant and ordinary cohomology of $X_h$ are
described using GKM-theory. The main goal of this paper is to show
the connection between the manifolds $X_h$ and the semisimple
Hessenberg varieties well-known in algebraic geometry. Both the
spaces $X_h$ and Hessenberg varieties form wonderful families of
submanifolds in the complete flag variety. There is a certain
symmetry between these families which can be generalized to other
submanifolds of the flag variety.
\end{abstract}

\maketitle


\section{Introduction}

A Hessenberg function is a function $h\colon[n]\to[n]$ such that
$h(i)\geqslant i$ and $h(i+1)\geqslant h(i)$. Given such a
function and a linear operator $S\colon\Co^n\to\Co^n$ one can
define a subvariety $Y_h$ of the variety $\Fl_n$ of complete
complex flags in $\Co^n$:
\[
\Hess_{h,S}=\{V_\bullet\in\Fl_n\mid SV_i\subset V_{h(i)}\},
\]
called Hessenberg variety. If $S$ is diagonalizable with distinct
eigenvalues, it is known that $Y_h=\Hess_{h,S}$ is nonsingular,
and its smooth type is independent of $S$. The manifold $Y_h$
carries an action of an algebraical torus $(\Co^\times)^n$ induced
by the standard action of $(\Co^\times)^n$ on $\Co^n$. Therefore
there is an action of a compact torus $T^n\subset(\Co^\times)^n$
on $Y_h$.

There is another natural space associated with the given
Hessenberg function $h$: the space $X_h$ of Hermitian matrices
which have a given simple spectrum
$\lambda=(\lambda_1,\ldots,\lambda_n)$ and staircase form
determined by $h$. The manifold of all Hermitian matrices with
spectrum $\lambda$ can be identified with complete flag manifold,
therefore the spaces $X_h$ form a family of subspaces of $\Fl_n$.
Conjugation by diagonal matrices determines a natural action of a
compact torus $T^n$ on $X_h$.

For the case $\hmin(i)=i+1$, $i=1,\ldots,n-1$, $h(n)=n$, the
manifold $Y_{\hmin}$ is known to be the toric variety
corresponding to root system of type $A_n$. The space $X_{\hmin}$
is the space of tridiagonal isospectral Hermitian matrices. This
space is known to be a quasitoric manifold over the permutohedron,
whose characteristic function is determined by a proper coloring
of facets (see details in \cite{DJ}). Two conclusions can be drawn
from this example. First, the manifolds $X_{\hmin}$ and
$Y_{\hmin}$ are different (see example \ref{exXYareDifferent}).
Second, the spaces $X_{\hmin}$ and $Y_{\hmin}$ are closely
related: their equivariant cohomology rings are isomorphic, and
their orbit spaces are combinatorially the same. This particular
case was studied in details in the work of Bloch, Flaschka, and
Ratiu \cite{BFR}.

In this paper we show a similar connection between $X_h$ and $Y_h$
in general. First we show that the space $X_h$ of staircase
isospectral matrices is a smooth manifold, and its smooth type is
independent of the spectrum $\lambda$ (see Theorem
\ref{thmXhSmoothEvenCells}). This statement is a folklore: the
essential idea of using the Toda flow is well-known, and certain
versions of this statement were proved in several papers (see e.g.
\cite{Tomei} for real tridiagonal matrices, and \cite{dMP} for
general real staircase matrices).

Applying Morse theory we show that odd degree cohomology of $X_h$
vanish. As a corollary, the manifold $X_h$ is equivariantly
formal, and its equivariant cohomology ring can be described by
GKM-theory. The corresponding results for the manifold $Y_h$ were
proved in \cite{Tym}.

Finally, we show the relation between $X_h$ and $Y_h$. Let $U(n)$
be the unitary group. The manifold $U(n)$ admits two actions of a
compact torus $T^n$ defined by left and right multiplication. Each
of these actions is free and the orbit space in both cases is a
manifold of complete complex flags in $\Co^n$.
\begin{equation}\label{eqTwoMaps}
\Fl_n\cong T^n\backslash
U(n)\stackrel{p}{\longleftarrow}U(n)\stackrel{q}{\longrightarrow}
U(n)/T^n\cong \Fl_n.
\end{equation}
The relation between $X_h$ and $Y_h$ is fairly simple and is given
by the following

\begin{thm}\label{thmInUnitary}
There exists a submanifold $Z_h\subseteq U(n)$ invariant under
both left and right actions of $T^n$ such that its left quotient
is $X_h$, and its right quotient is $Y_h$.
\end{thm}

This result easily implies the homeomorphism $X_h/T^n\cong
Y_h/T^n$, the isomorphism $H^*_T(X_h)\cong H^*_T(X_h)$ of
equivariant cohomology rings, and the coincidence of Betti numbers
of $X_h$ and $Y_h$, see Theorem \ref{thmXYconnection}.

The considerations of this paper make sense in the real case as
well. We have a real Hessenberg variety
\[
Y_h^\Ro=\{V_\bullet\in\Fl_n^\Ro\mid SV_i\subset V_{h(i)}\}
\]
sitting inside the manifold $\Fl_n^\Ro$ of complete flags in
$\Ro^n$. We also have the space $X_h^\Ro$ of isospectral staircase
real symmetric matrices. The smoothness of $X_h^\Ro$ follows from
the properties of Toda flow, see \cite{dMP}. The spaces $X_h^\Ro$
and $Y_h^\Ro$ are intimately related: there exists a submanifold
$Z_h^\Ro$ inside the orthogonal group $O(n)$ which covers both
$X_h^\Ro$ and $Y_h^\Ro$.

It should be mentioned here, that real manifolds $X_h^\Ro$ and
$Y_h^\Ro$ are different in general. In the work \cite{dMP}, where
the manifold $Y_h^\Ro$ was introduced, it was called a real
Hessenberg manifold. We consider this naming a certain inaccuracy.

One of the goals of the current paper is to show that there are
actually two types of manifolds which are different from
topological point of view. The manifold $Y_h$ is an algebraic
variety, while $X_h$ is not (in general). The manifold $X_h$ has
stably trivial tangent bundle, while $Y_h$ has not (again, in
general). The Betti numbers of $X_h$ and $Y_h$ coincide, however
the multiplicative structures of their cohomology rings may be
different.

Different methods are used to prove the smoothness of $X_h$ and
$Y_h$. Smoothness of $X_h$ follows from the properties of
generalized Toda flow. Smoothness of $Y_h$ follows from the
properties of Bialynicki-Birula decomposition. It seems that there
should exist a dynamics on $Z_h$ which covers the Toda flow (or
its generalizations) on $X_h$ and Bialynicki-Birula flow on $Y_h$.

It should be noted that smoothness of one of the spaces $X_h$,
$Y_h$ implies the smoothness of the other according to Theorem
\ref{thmInUnitary}. Indeed, we have a diagram \eqref{eqTwoMaps} of
two smooth fibrations. If $M$ is any $T$-invariant smooth manifold
in $U(n)/T^n$, then $q^{-1}(M)\subset U(n)$ is a smooth manifold,
and, therefore, $\tilde{M}=pq^{-1}(M)$ is also smooth. This
construction allows to make $X_h$ out of $Y_h$ and vice versa.

Following this recipe, for every $T$-invariant smooth submanifold
$M$ in a flag manifold one can construct its ``twin'' $\tilde{M}$.
This can be used to construct examples of well-behaved manifolds
having the same orbit spaces but different characteristic data.
The reasoning works in the real case as well.

\section{Manifolds of sparse isospectral matrices}\label{secGeneral}

\paragraph{Isospectral matrices}

Let $M_n$ denote the vector space of all Hermitian matrices of
size $n\times n$. We have $\dim_\Ro M_n=n^2$. For a given set
$\lambda = \{\lambda_1,\ldots,\lambda_{n}\}$ of pairwise distinct
real numbers consider the subset $M_{\lambda}\subset M_n$ of all
matrices which have eigenvalues $\{\lambda_1,\ldots,\lambda_{n}\}$
(assume $\lambda_1<\lambda_2<\cdots<\lambda_{n}$).

Let $U(n)$ be the group of unitary matrices and $T^{n}\subseteq
U(n)$ be the compact torus, that is the subgroup of diagonal
unitary matrices
\[
T^{n}=\left\{D=\diag(t_1,\ldots,t_{n}),
 t_i\in \Co, |t_i|=1\right\}
\]

Note that the group $U(n)$ acts on $M_{n}$ by conjugation.
Multiplying Hermitian matrix by $\sqrt{-1}$, we get a
skew-Hermitian matrix, therefore such an action can be identified
with the adjoint action of Lie group $U(n)$ on its tangent Lie
algebra. For a simple spectrum $\lambda$ the subset $M_\lambda$ is
identified with the principal orbit of the adjoint action.
Therefore $M_\lambda$ is diffeomorphic to a complete flag variety
$\Fl_{n}=U(n)/T^{n}$. Moreover, there is a trivial smooth
fibration
\[
p\colon M_n\setminus\Sigma\to C,
\]
where $\Sigma$ is the set of Hermitian matrices with multiple
eigenvalues, $C=\{(\lambda_1,\ldots,\lambda_n)\in\Ro^{n}\mid
\lambda_1<\cdots<\lambda_n\}$ is a Weyl chamber, and $p$ maps the
matrix to its eigenvalues in increasing order. The fiber of $p$
over the point $\lambda$ is the manifold $M_\lambda$. We have
$\dim_\Ro M_{\lambda}=n(n-1)$. The group $T^{n}$ acts on $M_{n}$
by conjugation: $A\mapsto DAD^{-1}$. In coordinate notation we
have
\[
(a_{ij})_{\substack{i=1,\ldots,n\\j=1,\ldots,n}}\mapsto
(t_it_j^{-1}a_{ij})_{\substack{i=1,\ldots,n\\j=1,\ldots,n}}
\]
Scalar matrices commute with every matrix $A$, therefore the
diagonal subgroup of torus acts non-effectively. The fixed points
of the action of $T^n$ on $M_{\lambda}$ are the diagonal matrices
with spectrum $\lambda$, i.e. the matrices of the form
$A_\sigma=\diag(\lambda_{\sigma(1)},\lambda_{\sigma(2)},\ldots,\lambda_{\sigma(n)})$
for all possible permutations $\sigma\in \Sigma_{n}$.

Let $\Gamma$ be a simple graph (that is a finite graph without
multiple edges and loops), which has the vertex set
$\{1,\ldots,n\}$, and edge set $E$. We associate with $\Gamma$ a
vector subspace in the set of all Hermitian matrices:
\[
M_\Gamma=\{A\in M_{n}\mid a_{ij}=0, \mbox{ if } \{i,j\}\notin E\}.
\]
Note that the action of $T^n$ by conjugation preserves the set
$M_\Gamma$. Also let
\[
M_{\Gamma,\lambda}=M_{\Gamma}\cap M_{\lambda}.
\]
The space $M_{\Gamma,\lambda}$ is called \emph{the space of
isospectral sparse matrices of type} $\Gamma$. The torus $T^n$
acts on $M_{\Gamma,\lambda}$ and we have

\begin{equation}\label{eqDimM}
\dim M_{\Gamma,\lambda}=2|E|
\end{equation}

\begin{ex}
If $\Gamma$ is the complete graph on $\{1,\ldots,n\}$, we have
$M_{\Gamma,\lambda}=M_\lambda\cong\Fl_n$.
\end{ex}

\begin{ex}
If $\Gamma$ is the graph with no edges, then $M_{\Gamma,\lambda}$
is the finite set of cardinality $(n+1)!$ consisting of all
diagonal matrices with spectrum $\lambda$.
\end{ex}

\begin{rem}\label{remBlockMatrices}
The case of disconnected graphs can be reduced to the connected
case. Let $\Gamma_1,\ldots,\Gamma_k$ be the connected components
of graph $\Gamma$, having vertex sets $A_1,\ldots,A_k\subset
\{1,\ldots,n\}$. Let $\Omega$ be the set of all possible
partitions of the set $\{\lambda_1,\ldots,\lambda_n\}$ into
nonintersecting subsets $S_i$ of cardinalities $|A_i|$,
$i=1,\ldots,k$. Then we have $M_{\Gamma,\lambda}=\bigsqcup_\Omega
\prod_{i=1}^k M_{\Gamma_i,S_i}$.
\end{rem}

In the following we consider only connected graphs.

\begin{probl}
Is it true that for each graph $\Gamma$ the space
$M_{\Gamma,\lambda}$ is a smooth submanifold whose diffeomorphism
type is independent of the spectrum $\lambda$?
\end{probl}

In the next section we describe the classical situation in which
the smoothness of $M_{\Gamma,\lambda}$ is proved using generalized
Toda flow.

\section{Staircase Hermitian matrices}\label{secStaircaseMatricesGeneral}

Let $h\colon [n]\to[n]$ be the function satisfying the condition
$h(i)\geqslant i$ for $0\leqslant i\leqslant n$, and
$h(i+1)\geqslant h(i)$ for $0\leqslant i\leqslant n-1$. Such
functions are called \emph{Hessenberg functions}. We write
Hessenberg function by listing its values:
$(h(1),h(2),\ldots,h(n))$.

For a Hessenberg function $h$ consider the vector subspace $M_h$
of sparse Hermitian matrices of the form
\[
M_h=\{A\in M_n\mid a_{ij}=0, \mbox{ if either } j>h(i), \mbox{ or
} i>h(j)\}
\]
Obviously, $M_h$ coincides with $M_{\Gamma_h}$, where $\Gamma_h$
is the graph on the set $\{1,\ldots,n\}$ with edge set
$E=\{(i,j)\mid i\leqslant j\leqslant h(i)\}$.

\begin{rem}
The graphs $\Gamma_h$ arising from Hessenberg functions are known
in combinatorics. According to \cite{Mert}, the $\Gamma$ is
isomorphic to $\Gamma_h$ for some Hessenberg function $h$ if and
only if $\Gamma$ is an \emph{indifferent graph} (otherwise called
\emph{proper interval graph}), see definition in \cite{Mert}. It
is known that any indifferent graph is chordal, but not the
converse.
\end{rem}

We call the Hessenberg function $h$ \emph{indecomposable} if it
satisfies the relation $h(i)>i$ for all $i=1,\ldots,n-1$. If $h$
is not indecomposable, the corresponding graph $\Gamma_h$ is
disconnected, therefore the corresponding space breaks down into
the disjoint union of products of smaller spaces, see remark
\ref{remBlockMatrices}. Therefore it is sufficient to consider
only indecomposable Hessenberg functions.

\begin{rem}
Note that the number of indecomposable Hessenberg functions
$h\colon[n]\to[n]$ equals the $n$-th Catalan number
$C_n=\frac{1}{n+1}{2n\choose n}$.
\end{rem}

For a Hessenberg function $h\colon[n]\to [n]$, consider the
function $g\colon[n]\to [n]$
\[
g(i)=\min\{j\mid h(j)\geqslant i\}
\]
It can be seen that $g(i)\leqslant i$ for all $i\in[n]$ and
\begin{equation}\label{eqHGinterchange}
h(i)\geqslant j \Longleftrightarrow i\geqslant g(j).
\end{equation}

The element $A_{i,j}$ of matrix $A\in M_h$ is allowed to be
nonzero only if $g(i)\leqslant j\leqslant h(i)$.

\begin{ex}
$6\times 6$-matrices, corresponding to Hessenberg function
$(3,3,5,6,6,6)$ have the form
\[
\begin{pmatrix}
\ast & \ast & \ast & 0 & 0 & 0 \\
\ast & \ast & \ast & 0 & 0 & 0 \\
\ast & \ast & \ast & \ast & \ast & 0 \\
0 & 0 & \ast & \ast & \ast & \ast \\
0 & 0 & \ast & \ast & \ast & \ast \\
0 & 0 & 0 & \ast & \ast & \ast
\end{pmatrix}
\]
The values of dual function $g$ are $1,1,1,3,3,4$.
\end{ex}

\begin{ex}
For the function
\begin{equation}\label{eqPlusOneFunctionH}
\hmin(i)=i+1, i=1,\ldots,n-1,\qquad \hmin(n)=n.
\end{equation}
the space $M_{\hmin}$ is the space of tridiagonal Hermitian
matrices.
\end{ex}

Consider the space
\[
X_h=M_{h,\Lambda}=M_h\cap M_\Lambda,
\]
that is the set of staircase Hermitian matrices having given
simple spectrum $\lambda_1<\cdots<\lambda_n$. The compact torus
$T^n$ acts on $M_{h,\Lambda}$ as before, with the diagonal
$\Delta(T^1)$ acting non-effectively. If $h$ is indecomposable,
then the induced action of $T^n/\Delta(T^1)$ is effective. Let
$d=d(h)=\sum_{i=1}^n(h(i)-i)$.

Consider the real version of these constructions. Let $X_h^\Ro$ be
the set of staircase symmetric real matrices with the given simple
spectrum. The group $\Zt^n$ acts on this space.

\begin{thm}\label{thmXhSmoothEvenCells}
The space $X_h$ is a smooth $2d$-dimensional manifold, whose
smooth type is independent of $\Lambda$. The space $X_h^\Ro$ is a
smooth $d$-dimensional manifold, whose smooth type is independent
of $\Lambda$. Odd cohomology of the space $X_h$ vanish.
\end{thm}

The real case of this statement was proved in the work \cite{dMP},
where, in particular, the homology of $X_h^\Ro$ with
$\Zt$-coefficients were calculated.

Before proving the theorem we give some of its important
consequences. Recall the notion of Borel construction. Let $X$ be
a space with the action of a compact torus $T\cong T^k$. The Borel
construction of $X$ is the space $X_T=X\times_TET$, where $ET\to
BT$ is the classifying bundle of the group $T$. Note that
$BT\simeq (\CP^\infty)^k$. The natural map $p\colon X_T\to BT$ is
the Serre fibration with fiber $X$. Equivariant cohomology ring of
$X$ is the ring $H^*_T(X)=H^*(X_T)$. In the following we assume
that cohomology have coefficients in $\Zo$, unless stated
otherwise. The induced map $p^*\colon H^*(BT)\to H^*(X_T)$ makes
$H^*_T(X)=H^*(X_T)$ a module over the ring of polynomials
\[
H^*(BT)=H^*((\CP^\infty)^k)\cong \Zo[v_1,\ldots,v_k],\qquad \deg
v_i=2.
\]

\begin{defin}\label{definEquivFormal}
A space $X$ is called equivariantly formal if the Serre spectral
sequence of the fibration $p$
\[
E_2^{p,q}\cong H^p(BT;H^q(X))\cong H^p(BT)\otimes
H^q(X)\Rightarrow H^*_T(X)
\]
collapses at the second page.
\end{defin}

\begin{rem}\label{remOddIsZero}
Note that vanishing of odd cohomology of $X$ implies that
\[
E_2^{p,q}=0, \mbox{ if } p \mbox{ or } q \mbox{ are odd.}
\]
In this case all higher differential in the spectral sequence
vanish by trivial reasons, hence $X$ is equivariantly formal.
\end{rem}

\begin{cor}
The space $X_h$ with the action of the torus $T$ is equivariantly
formal.
\end{cor}

In equivariantly formal case the ordinary cohomology ring can be
obtained from the equivariant cohomology ring by changing the
coefficient ring:
\[
H^*(X)\cong H^*_T(X)\otimes_{H^*(BT)}\Zo=H^*_T(X)/(p^*(H^+(BT))).
\]
In many situations equivariant cohomology algebra of equivariantly
formal manifolds can be computed using GKM-theory, see \cite{GKM}
or \cite{Kur}.

\begin{prop}\label{propXhIsGKM}
The space $X_h$ is a GKM-manifold.
\end{prop}

Recall the definition of the Hessenberg variety. Let $h\colon
[n]\to [n]$ be a Hessenberg function, $\Lambda$ the diagonal
$n\times n$-matrix with pairwise distinct values at diagonal (we
will assume these numbers real, however it is not necessary in the
definition). Let $\Fl_n$ be the complete flag variety
\[
\Fl_n=\{V_\bullet=(V_0\subset V_1\subset\cdots\subset V_n)\mid
V_i\subseteq \Co^n, \dim V_i=i\}.
\]

\begin{defin}
The (semisimple) Hessenberg variety corresponding to $h$ (and the
matrix $\Lambda$), is the subset $Y_h\subset \Fl_n$,
\[
Y_h=\{V_\bullet\in \Fl_n\mid SV_i\subseteq V_{h(i)}\}.
\]
\end{defin}

The algebraical torus $(\Co^\times)^n$ acts on $\Co^n$, inducing
the action of $(\Co^\times)^n$ on $\Fl_n$. The subvariety $Y_h$ is
preserved by this action.

It is known (see \cite{Tym} or \cite{AHM} and references therein),
that $Y_h$ is a smooth variety of complex dimension $d$, its odd
cohomology vanish, and $Y_h$ is a GKM-manifold with respect to the
action of a compact torus $T^n\subset (\Co^\times)^n$.

\begin{thm}\label{thmXYconnection}
For the manifold $X_h$ of isospectral staircase matrices and the
Hessenberg variety $Y_h$ the following hold
\begin{enumerate}
\item The orbit spaces are homeomorphic: $X_h/T\cong Y_h/T$,
moreover the homeomorphism preserves the orbit type filtration.
\item The unlabeled GKM-graphs of GKM-manifolds $X_h$ and $Y_h$ coincide.
\item Equivariant cohomology are isomorphic as graded rings:
\[
H^*_T(X_h)\cong H^*_T(Y_h).
\]
\item The Betti numbers of $X_h$ and $Y_h$ coincide:
\[
H^{2i}(X_h)\cong H^{2i}(Y_h)
\]
\end{enumerate}
\end{thm}

\begin{rem}\label{remXYEquivDifferent}
For the Hessenberg function $\hm\colon[n]\to[n]$, $\hm(i)=n$,
$\forall i\in[n]$, the space $X_{\hm}$ is the space of all
isospectral Hermitian matrices, and $Y_{\hm}$ is the whole flag
variety $\Fl_n$. Therefore in this case we have $X_h\cong Y_h$.

We should remark that in general the manifolds $X_h$ and $Y_h$ are
different. For a function $\hmin$ (see \eqref{eqPlusOneFunctionH})
the space $X_{\hmin}$ is the manifold of isospectral tridiagonal
Hermitian matrices. It is known \cite{Tomei,BFR,DJ}, that this
manifold is a quasitoric manifold over the permutohedron, and its
characteristic function is determined by a proper coloring of
facets of permutohedron. On the other hand, it is known
\cite{BFR}, that Hessenberg variety $Y_{\hmin}$ is the toric
variety corresponding to root system of type $A_n$. Therefore,
$Y_{\hmin}$ is also a quasitoric manifold over the permutohedron,
however its characteristic function is determined by normal
vectors to facets of permutohedron. The characteristic functions
are non-equivalent thus, at least from equivariant point of view,
the manifolds $X_{\hmin}$ and $Y_{\hmin}$ are different.
\end{rem}

\begin{ex}\label{exXYareDifferent}
In case $n=3$ the 4-manifolds $X_{\hmin}$ and $Y_{\hmin}$ can be
explicitly described. The manifold $Y_{\hmin}$ is a toric manifold
over the regular hexagon. Such hexagon can be obtained from a
triangle by cutting its vertices. Therefore $Y_{\hmin}$ is the
blow up of $\CP^2$ at 3 points, hence, up to diffeomorphism,
$Y_{\hmin}\cong \CP^2\hash3\overline{\CP^2}$.

The manifold $X_{\hmin}$ is quasitoric, and its characteristic
pair is shown on fig.\ref{pictConSum}. It can be seen that the
characteristic pair is the connected sum of two squares along the
vertex. Therefore, $X_h\cong (S^2\times S^2)\hash (S^2\times
S^2)$. The manifolds $X_{\hmin}$ and $Y_{\hmin}$ are
non-diffeomorphic: for example they have different signatures.
\end{ex}

\begin{figure}[h]
\begin{center}
\includegraphics[scale=0.3]{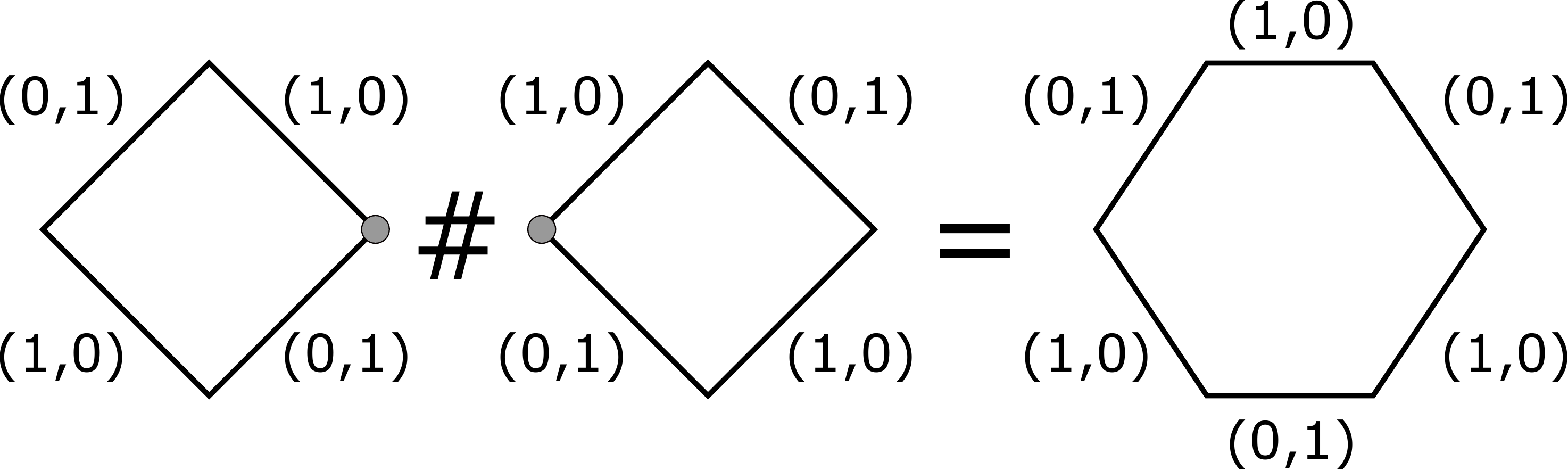}
\end{center}
\caption{Characteristic pair for the hexagon}\label{pictConSum}
\end{figure}


\section{Generalized Toda flow}\label{secStaircaseMatrAreManifolds}

To prove Theorem \ref{thmXhSmoothEvenCells} we use the properties
of isospectral flows. These properties are briefly described
below.

Let $L\in M_n$ be a Hermitian matrix, $L=L_++L_0+L_-$, where
$L_+,L_0,L_-$ are strictly upper triangular, diagonal, and
strictly lower triangular matrices respectively. So far
$L_+=\overline{L}_-^t$. Consider the skew Hermitian matrix
$P=P(L)=L_--L_+$. The generalized Toda flow is the dynamical
system
\begin{equation}\label{eqGenToda}
\dot{L}=[L,P(L)]=LP-PL.
\end{equation}
The following holds.

\begin{enumerate}
\item Since $L$ is Hermitian, and $P$ is skew Hermitian,
the commutator $[L,P]$ is Hermitian. Therefore the subspace $M_n$
of Hermitian matrices is preserved by the flow~\eqref{eqGenToda}.

\item The standard reasoning (see, e.g. \cite{Moser})
shows that flows of the form $\dot{L}=[L,A(L)]$ preserve the
spectrum of $L$. Therefore the subset $M_\lambda$ is preserved by
the flow~\eqref{eqGenToda}.

\item Finally, it can be checked by a direct computation that $L$
has a staircase form, determined by a Hessenberg function $h$,
i.e. $L\in M_h$, then $P(L)\in M_h$ and moreover $[L,P]\in M_h$.
Therefore the subset $M_h$ together with $X_h=M_h\cap M_\lambda$
are preserved by the flow \eqref{eqGenToda}.

\item The flow commutes with the action of torus $T^{n-1}$.
Indeed, the action on both $L$ and $P$ is the conjugation by
diagonal matrix $D$, therefore we have:
\[
[DLD^{-1},DPD^{-1}]=D[L,P]D^{-1}=D\dot{L}D^{-1}=(DLD^{-1})^\cdot
\]

\item The set of real symmetric matrices is invariant under the
flow \eqref{eqGenToda}. Natural analogs of the previous statements
hold in the real case.
\end{enumerate}

\begin{ex}
For the Hessenberg function $\hmin$ in the real case, the flow
\eqref{eqGenToda} is the ordinary non-periodic Toda flow on the
space of tridiagonal symmetric matrices. This is the classical
example of integrable dynamical system, see \cite{Moser}.
\end{ex}

\begin{prop}[{see.\cite[Thm.3.1]{Nan}}]\label{propAsymptDiagonal}
Let $L(t)$ be a trajectory of the flow \eqref{eqGenToda}. If
$t\to+\infty$ or $t\to-\infty$, the solution $L(t)$ tends to a
diagonal matrix.
\end{prop}

Since the spectrum $\lambda=(\lambda_1<\cdots<\lambda_n)$ is
preserved by the flow, the limit matrix has the form
\begin{equation}\label{eqDiagonalMatr}
A_\sigma=\diag(\lambda_{\sigma(1)},\ldots,\lambda_{\sigma(n)})
\end{equation}
for some permutation $\sigma\in\Sigma_n$.

We will prove the smoothness of $X_h$ following the classical idea
used by Tomei in tridiagonal case \cite{Tomei}. Let the diagonal
elements of the matrix $L\in X_h$ be denoted $a_1,\ldots,a_n$, and
above-diagonal elements are $b_{i,j}$, where $i<j\leqslant h(i)$.

\begin{lem}[{see. \cite[Lm 2.1]{Tomei}}]
Let $A_\sigma\in M_h$ be a diagonal matrix with distinct
eigenvalues. Then the mapping $\psi\colon U\to
\Ro^n\times\Co^{d(h)}$, $L\mapsto (\{\lambda_i\},\{b_{i,j}\mid
i<j\leqslant h(j)\})$ is a local diffeomorphism on some
neighborhood $U\subset M_h$ of a matrix $A_\sigma$.
\end{lem}

Therefore, the map $M_h\to \Ro^n$, $L\mapsto
(\lambda_1,\ldots,\lambda_n)$ is a smooth submersion in
neighborhoods of diagonal matrices. Now we prove that the spaces
$X_h$ are smooth manifolds, and their smooth types are independent
of simple spectrum $\lambda$, which is the first part of Theorem
\ref{thmXhSmoothEvenCells}.

\begin{proof} Let $L\in M_h$ be an arbitrary staircase Hermitian matrix with
simple spectrum. Consider the trajectory of generalized Toda flow
having initial value $L(0)=L$. According to proposition
\ref{propAsymptDiagonal}, for some $t$ we have $L(t)\in U$, where
$U$ is a neighborhood of diagonal matrix. The flow determines a
diffeomorphism $\phi_t$ from a neighborhood $V\subset M_h$ of
matrix $L$ to the neighborhood $U'\subset U$ of matrix $L(t)$. The
composed map
\[
\psi\circ\phi_t\colon V\to \Ro^n,
\]
sends every matrix to its spectrum according to isospectral
property of the flow. Since this map is the composition of a
diffeomorphism and a submersion, it is a submersion. Therefore the
map
\[
p\colon M_h\setminus \Sigma\to
C=\{(\lambda_1,\ldots,\lambda_n)\in\Ro^{n}\mid
\lambda_1<\cdots<\lambda_n\}, \qquad L\mapsto \mbox{ spectrum of
}L,
\]
is a smooth submersion at all points, where $\Sigma$ is the set of
matrices with multiple eigenvalues. Thus $p$ is a smooth
fibration, and all its fibers $p^{-1}(\lambda)=M_h\cap
M_\lambda=X_h$ are diffeomorphic.
\end{proof}

The generalized Toda flow determines the dynamical system on a
manifold $X_h$. According to Proposition \ref{propAsymptDiagonal}
the equilibria points of this system are the diagonal matrices
$A_\sigma$, moreover, every trajectory tends to one of these
matrices if $t\to \pm\infty$. Therefore, the system has no closed
periodical orbits except equilibria points. To prove the vanishing
of odd cohomology we use Morse theory.

\begin{lem}
Let $\hm(i)=n$, and consequently $X_{\hm}$ is diffeomorphic to the
complete complex flag variety $\Fl_n$. The generalized Toda flow
is a gradient flow on $X_{\hm}$.
\end{lem}

\begin{proof}
The proof repeats the idea of \cite{ChShS}, where it was developed
in the real case. At first note that Toda flow can be written in a
simple way on the level of unitary matrices. The solution $L(t)$
of the flow \eqref{eqGenToda} has the form $L(t)=U(t)\Lambda
U(t)^{-1}$, where $U(t)$ is the solution of the equation
\[
\dot{U}=-P(L)U,\qquad P(L)=P(U\Lambda U^{-1})=(U\Lambda
U^{-1})_--(U\Lambda U^{-1})_+.
\]
To prove the lemma, we need to specify the Riemannian metric and
the function on $U(n)/T^n\cong X_{\hm}$ such that the generalized
Toda flow is the gradient of this function. We determine the
function $F$ on the manifold $U(n)/T^n$ by setting
$F(U)=a_1+2a_2+\cdots+na_n$, where $(a_1,\ldots,a_n)$ is the
diagonal of the Hermitian matrix $L=U\Lambda U^{-1}\in X_{\hm}$.
If we denote the diagonal matrix $\diag(1,2,\ldots,n)$ by $N$, we
have $F(U)=\Tr(LN)=\Tr(U\Lambda U^{-1}N)$.

Now we define a Riemannian metric on $U(n)/T^n$. At first we
define a bilinear form on the tangent space $\uu(n)$ of the group
$U(n)$. The tangent algebra $\uu(n)$ consists of skew Hermitian
matrices of size $n$. Let $J$ be a linear operator on $\uu(n)$,
which multiplies each element $\omega_{ij}$ of skew Hermitian
matrix $\Omega\in\uu(n)$ by $|i-j|$. Hence, the $k$-th lower and
upper diagonals of a matrix are multiplied by $k$, in particular
its main diagonal vanishes. Strictly speaking, $J$ does not have
an inverse map. Nevertheless, it will be convenient to define
$J^{-1}$ as the operator acting on the subspace of all skew
Hermitian matrices with zeroes at main diagonal. The operator
$J^{-1}$ divides the elements of $k$-th diagonal by $k$. For each
matrix $\Omega\in \uu(n)$ having zeroes at main diagonal we have
$\Omega=JJ^{-1}(\Omega)$.

There is a standard Killing form on the space $\uu(n)$
\[
\langle \Omega_1,\Omega_2\rangle=-\Tr(\Omega_1\Omega_2).
\]
We consider another bilinear form by setting
\[
\langle \Omega_1,\Omega_2\rangle_J=\langle
\Omega_1,J(\Omega_2)\rangle=-\Tr(\Omega_1J(\Omega_2))
\]
The form $\langle\cdot,\cdot\rangle_J$ is well defined on
$\uu(n)$, however it is degenerate. Its kernel consists of
diagonal matrices, since such matrices are annihilated by $J$.
Therefore, the form $\langle\cdot,\cdot\rangle_J$ is a
nondegenerate positive-definite form on $\uu(n)/\ttt(n)$, that is
on the tangent space of $U(n)/T^n$. We spread the form
$\langle\cdot,\cdot\rangle_J$ over $U(n)/T^n$ by left transitive
action of $U(n)$. This defines the Riemannian metric on
$U(n)/T^n\cong X_{\hm}$.

Now we show that the gradient flow of the function $F$ with
respect to metric $\langle\cdot,\cdot\rangle_J$ coincides with the
generalized Toda flow. Let $U_0\in U(n)$ and $L_0=U_0\Lambda
U_0^{-1}$. Consider the infinitesimal shift $U=(1+\Omega)U_0$,
$\Omega\in \uu(n)$. We have
\begin{multline}
F(U)=\Tr((1+\Omega)U_0\Lambda
U_0^{-1}(1-\Omega)N)=\\=F(U_0)+\Tr(\Omega U_0\Lambda
U_0^{-1}N)-\Tr(U_0\Lambda U_0^{-1}\Omega N)+o(\Omega)=\\=
F(U_0)+\Tr(\Omega[U_0\Lambda
U_0^{-1},N])+o(\Omega)=\\=F(U_0)+\Tr(\Omega\cdot
J(J^{-1}([L_0,N])))+o(\Omega)=\\=F(U_0)-\langle\Omega,J^{-1}([L_0,N])\rangle_J+o(\Omega).
\end{multline}
The transition from third to forth line is correct since $[L_0,N]$
is a matrix with zeroes at diagonal. A direct check shows that
$-J^{-1}([L_0,N])=P(L_0)=(L_0)_--(L_0)_+$. Therefore,
$\grad_{\langle\cdot,\cdot\rangle_J} F$ at point $[U]\in U(n)/T^n$
coincides with $-P(U\Lambda U^{-1})$, which means that it
determines the generalized Toda flow.
\end{proof}

The equilibria points $A_\sigma$ of the Toda flow are singular
points of the function $F$. Let
\[
W^s_X(A)=\{L=L(0)\in X\mid L(t)\to A \mbox{ if }t\to+\infty\};
\]
\[
W^u_X(A)=\{L=L(0)\in X\mid L(t)\to A \mbox{ if }t\to-\infty\}
\]
denote the stable and unstable manifolds of the equilibrium point
$A$ of the flow $\dot{L}=f(L)$ on a manifold $X$. For the
generalized Toda flow \eqref{eqGenToda} on the manifold $\Fl\cong
X_{\hm}$ and its submanifolds $X_h$ we have
\[
W^s_{X_h}(A_\sigma)=W^s_{\Fl}(A_\sigma)\cap X_h,
\]
for each equilibrium point $A_\sigma$ of the form
\eqref{eqDiagonalMatr}.

\begin{lem}\label{lemEquilibriumPoints}
All equilibria points $A_\sigma$ of the flow \eqref{eqGenToda} on
$X_h$ have hyperbolic type. The dimension of the stable manifold
at a point $A_\sigma$ is equal to
\[
\dim A_\sigma=2\sharp\{1\leqslant i< j\leqslant h(i)\mid
\sigma(i)>\sigma(j)\}
\]
\end{lem}

\begin{proof}
We will need a coordinate form of the flow $\dot{L}=[L,P]$. As
before, let $a_i$, $i\in [n]$ be the diagonal entries of $L\in
M_h$, and $b_{i,j}$ be its above diagonal entries. The staircase
condition implies a restriction $b_{i,j}=0$ if $j>h(i)$. We have
$b_{i,j}=\overline{b_{j,i}}$, $P_{i,j}=-b_{i,j}$ if $i<j$,
$P_{i,j}=b_{i,j}$ if $i>j$, $P_{i,i}=0$. Let us compute an
off-diagonal element of the matrix $[L,P]$. For $i<j$ we have
\[
(LP)_{ij} =
-\sum_{k<i}b_{ik}b_{kj}-a_{i}b_{ij}-\sum_{i<k<j}b_{ik}b_{kj}+
\sum_{k>j}b_{ik}b_{kj}
=-a_ib_{ij}-\sum_{k<j}\overline{b_{ki}}b_{kj}-\sum_{i<k<j}b_{ik}b_{kj}+
\sum_{k>j}b_{ik}\overline{b_{jk}};
\]
\[
(PL)_{ij}=\sum_{k<i}b_{ik}b_{kj}-\sum_{i<k<j}b_{ik}b_{kj}-b_{ij}a_j-
\sum_{k>j}b_{ik}b_{kj}=-b_{ij}a_j+\sum_{k<j}\overline{b_{ki}}b_{kj}-\sum_{i<k<j}b_{ik}b_{kj}-
\sum_{k>j}b_{ik}\overline{b_{jk}}.
\]
\[
[L,P]_{ij}=b_{ij}(a_j-a_i)+2\sum_{j<k\leqslant
h(i)}b_{ik}\overline{b_{jk}}-2\sum_{g(j)\leqslant
k<i}\overline{b_{ki}}b_{kj}.
\]
Hence
\begin{equation}\label{eqTodaCoordin}
\dot{b}_{ij}=b_{ij}(a_j-a_i)+2\sum_{j<k\leqslant
h(i)}b_{ik}\overline{b_{jk}}-2\sum_{g(j)\leqslant
k<i}\overline{b_{ki}}b_{kj}
\end{equation}

The equilibrium point $A_\sigma$ corresponds to
$a_i=\lambda_{\sigma(i)}$, $b_{ij}=0$. To linearize the system
\eqref{eqTodaCoordin} in the neighborhood of $A_\sigma$ we assume
all $b_{ij}$ small and discard all terms of order $>1$ at the
right hand side of \eqref{eqTodaCoordin}. The linearized system
has the form
\begin{equation}\label{eqLinearization}
\dot{b}_{ij}=b_{ij}(\lambda_{\sigma(j)}-\lambda_{\sigma(i)}),\qquad
i<j\leqslant h(i).
\end{equation}
It can be seen that $2d(h)\times 2d(h)$-matrix of the linearized
system is diagonal (the doubling takes place since each number
$b_{ij}$ has real and imaginary component). Eigenvalues of the
linearized system have the form
$\lambda_{\sigma(j)}-\lambda_{\sigma(i)}$. All of them are either
positive or negative. Therefore the equilibrium point $A_\sigma$
has hyperbolic type. The dimension of the stable manifold equals
the number of negative eigenvalues of linearized system, i.e.
\[
\dim W^s_{X_h}(A_\sigma)=2\sharp\{i<j\leqslant h(i)\mid
\lambda_{\sigma(j)}-\lambda_{\sigma(i)}<0\}= 2\sharp\{i<j\leqslant
h(i)\mid \sigma(j)<\sigma(i)\}.
\]
The last identity holds since the order of real numbers
$\lambda_i$ coincides with the order of their indices.
\end{proof}

\begin{cor}\label{corMorseEven}
The function $F$ restricted to $X_h\subset X_{\hm}$ is a Morse
function. All its singular points have even indices. The Morse
complex is concentrated in even degrees thus its differential
vanishes. Odd homology and cohomology of $X_h$ vanish.
\end{cor}

The corollary \ref{corMorseEven} completes the proof of Theorem
\ref{thmXhSmoothEvenCells}.

\begin{rem}
It was shown in \cite{ChShS} that in the case of real symmetric
isospectral matrices $M_\lambda^\Ro\cong\Fl_n^\Ro$, the
decomposition into stable (resp. unstable) manifolds coincides
with decomposition of real flag manifold into Bruhat cells (resp.
dual Bruhat cells). Since Bruhat and dual Bruhat cells are
transversal, the generalized Toda flow is a Morse--Smale system on
$M_\lambda^\Ro$. All the ideas of the work \cite{ChShS} work in
the complex case as well.
\end{rem}

\section{Torus action}\label{secStaircaseMatrGKM}

In this section we study the action of a compact torus on $X_h$ in
detail. First recall the definition of a GKM-manifold and
surrounding theory. Let $T=T^k$ be a compact torus. The character
lattice $N=\Hom(T,S^1)\cong \Zo^k$ is naturally identified with
the group $H^2(BT;\Zo)$. For a character $\alpha\in\Hom(T,S^1)$
let $V(\alpha)$ denote the corresponding representation on
$\Co^1$:
\[
V(\alpha)\colon T\to \Hom(\Co,\Co),\qquad t\cdot z=\alpha(t)z.
\]

\begin{defin}[Goresky--Kottwitz--MacPherson]
A compact orientable manifold $X$, $\dim X=2m$ with a smooth
action of a compact torus $T=T^k$ is called \emph{GKM-manifold} if
it satisfies the following conditions:
\begin{enumerate}
\item $X$ is equivariantly formal;
\item The fixed point set $X^T$ of the action is finite;
\item The weights of the tangent representation of a torus at each
fixed point $x\in X^T$ are pairwise non-collinear:
\[
T_xX=\oplus_{i=1}^mV(\alpha_i)\quad\Rightarrow\quad
\alpha_i,\alpha_j\in \Zo^k \mbox{ non-collinear if }i\neq j
\]
\item Each two-dimensional submanifold of $X$, which is preserved
by $T$ and consists of no more than one-dimensional orbits,
contains a fixed point.
\end{enumerate}
\end{defin}

Conditions 2-4 in the definition guarantee that equivariant
1-skeleton of a GKM-manifold $X$ (i.e. the set of no more than
one-dimensional torus orbits) consists of a set of fixed points,
with some pairs of fixed points connected by $T$-invariant
2-spheres. The action of $T$ on each such sphere is given by a
character $\alpha\in N$ determined up to sign:
\[
T^k\circlearrowright S^2=\CP^1,\qquad
t[z_0;z_1]=[z_0;\alpha(t)z_1]
\]
Therefore one can associate a GKM-graph $G(X)=(V,E,\alpha)$ with a
GKM-manifold $X$. The vertex set $V$ of $G(X)$ is the set of fixed
points of the action; two vertices $p,q\in V$ determine an edge
$\{p,q\}\in E$, if there is an invariant 2-sphere in $X$ between
$p$ and $q$. Moreover, each edge $\{p,q\}\in E$ of a GKM graph
$G(X)$ has a label $\alpha_{pq}\in N\cong\Zo^k$, which encodes the
character of the $T$-action on the corresponding sphere.

There is an abstract definition of a GKM-graph, which axiomatizes
the properties of the graphs $G(X)$, however we will not need this
definition. Details can be found in \cite{Kur} or~\cite{GKM}. The
equivariant and ordinary cohomology rings of a GKM-manifold can be
extracted from its GKM-graph according to the following result.

\begin{thm}[Goresky--Kottwitz--MacPherson]
Let $X$ be a GKM-manifold and $G(X)=(V,E,\alpha)$ its GKM-graph.
There is an isomorphism of $H^*(BT;\Zo)$-algebras:
\[
H_T^*(X;\Zo)\cong \{\phi\colon V\to H^*(BT;\Zo)\mid
\phi(p)\equiv\phi(q)\mod (\alpha_{pq})\forall pq\in E\},
\]
where the character $\alpha_{pq}$ is considered as an element of
$H^2(BT;\Zo)$. For the cohomology ring we have
\[
H^*(X;\Zo)\cong H^*_T(X;\Zo)\otimes_{H^*(BT;\Zo)}\Zo,
\]
according to equivariant formality of $X$.
\end{thm}

This theorem gives a description of equivariant cohomology as a
submodule inside the free module
\[
\{\phi\colon V\to H^*(BT;\Zo)\}\cong H^*(BT;\Zo)^{|V|}.
\]
Now we prove Proposition \ref{propXhIsGKM}, which states that
$X_h$ is a GKM-manifold.

\begin{proof}
Equivariant formality is already proved (it follows from Theorem
\ref{thmXhSmoothEvenCells} and remark \ref{remOddIsZero}). Let
$t=(t_1,\ldots,t_n)\in T^n$. The torus action on $X_h$ is given by
$a_i\mapsto a_i$, $b_{ij}\mapsto t_it_j^{-1}b_{ij}$. Fixed points
of this action are the diagonal matrices: $X_h^T=\{A_\sigma,
\sigma\in\Sigma_n\}$ which is a finite set.

Consider the characters $\epsilon_i\in\Hom(T^n;U(1))$,
$\epsilon_i((t_1,\ldots,t_n))=t_i$, and let
$\epsilon_{ij}=\epsilon_i-\epsilon_j$. From the linearization
\eqref{eqLinearization} it can be seen that the tangent
representation of a torus $T^n$ in a fixed point $A_\sigma$ has
the form
\[
T_{A_\sigma}X_h=\bigoplus_{i<j\leqslant h(i)}V(\epsilon_{ij})
\]
Hence the weights of tangent representation are pairwise
non-collinear.

Condition (4) in the definition of GKM-manifold is quite
technical: usually it is not checked at all. For complex
GKM-manifolds it holds automatically according to
Bialynicki-Birula method, \cite{BB} (the dynamics $\lim_{t\to0}tx$
allows to reach a fixed from any given point $x$ on a manifold, a
generic algebraic subgroup $\Co^\times\subset (\Co^\times)^n$ and
$t\subset (0,1]\subset \Co^\times$).

The condition (4) is checked for $X_h$ by the following simple
consideration. Assume the condition violates, i.e. we have a
closed surface in $X_h$ consisting entirely of 1-dimensional
orbits and no fixed points. Since $X_h$ is a $T$-invariant
submanifold of $\Fl_n$, this surface lies in $\Fl_n$ as well.
However, it the complete flag variety $\Fl_n$ is a GKM-manifold
(see \cite{GHZ}) so the condition (4) holds for this space.
\end{proof}

It is easily checked that fixed points $A_\sigma$ and $A_\tau$ of
the manifold $X_h$ are connected by a $T$-invariant 2-sphere if
and only if
\[
\tau=\sigma\cdot(i,j)\quad i<j\leqslant h(i) \mbox{ либо }
j<i\leqslant h(j),
\]
where $(i,j)$ is a transposition. Such sphere consists of matrices
of the form
\[
\begin{tikzpicture}[decoration=brace] \matrix (m) [matrix of
math nodes,left delimiter={(},right delimiter={)}] {
\lambda_{\sigma(1)} & \cdots &  & \cdots & 0\\
\vdots & \ast & 0 & \ast & \vdots\\
 & 0 & \ddots &  0 & \\
\vdots & \ast & 0 & \ast & \vdots \\
0 & \cdots &  & \cdots & \lambda_{\sigma(n)} \\
};

\draw (m-1-2.north) node [yshift=6pt] {$i$}; \draw (m-1-4.north)
node [yshift=6pt] {$j$};

\draw (m-2-1.west) node [xshift=-6pt] {$i$}; \draw (m-4-1.west)
node [xshift=-6pt] {$j$};
\end{tikzpicture}
\]
where the block in the intersection of $i$-th and $j$-th rows and
columns has eigenvalues $\lambda_{\sigma(i)}=\lambda_{\tau(j)}$
and $\lambda_{\sigma(j)}=\lambda_{\tau(i)}$, and all other
eigenvalues are distributed along the diagonal according to the
permutation $\sigma$. GKM-theorem gives the following result.

\begin{prop}\label{propEquivCohomXh}
There is an isomorphism of $H^*(BT^n;\Zo)$-algebras
\[
H^*_T(X_h;\Zo)\cong\left\{\phi\colon \Sigma_n\to H^*(BT^n;\Zo)\mid
\pbox{\textwidth}{$\phi(\sigma)\equiv\phi(\tau)\mod
(\epsilon_i-\epsilon_j)$\\ if $\tau=\sigma\cdot(i,j)$ and
$i<j\leqslant h(i)$ or $j<i\leqslant h(j)$ }\right\}
\]
\end{prop}

Now we prove Theorem \ref{thmXYconnection} about the connection
between $X_h$ and Hessenberg variety~$Y_h$. In the first point we
also prove Theorem \ref{thmInUnitary}.

\begin{proof}
(1) Let us show that $X_h/T\cong Y_h/T$. We develop the idea of
\cite{BFR} used to relate the space of isospectral tridiagonal
matrices and toric variety of type $A_n$. Given an arbitrary
unitary matrix $U\in U(n)$ consider the flag
\[
F^U_\bullet=(F^U_1\subset F^U_2\subset\cdots\subset F^U_n),
\]
where $F^U_i\subset \Co^n$ is the subspace spanned by first $i$
columns of $U$. Let $E_\bullet=(E_1\subset E_2\subset\cdots\subset
E_n)$ be the fixed flag of coordinate subspaces, $E_i=\langle
e_1,\ldots,e_i\rangle$. Then we have $F^U_i=U(E_i)$. Vice versa,
each flag $V_\bullet$ determines a matrix $U$ such that
$F^U_\bullet=V_\bullet$, however this matrix is defined up to
right multiplication by diagonal matrices $\diag(t_1,\ldots,t_n)$,
$|t_i|=1$.

Note that $U(n)$ has both right and left free action of the
maximal torus $T^n$ (by right and left multiplication). Consider
the submanifold
\[
Z_h=\{U\in U(n)\mid U^{-1}\Lambda U\in M_h\},
\]
where $\Lambda=\diag(\lambda_1,\ldots,\lambda_n)$. The manifold
$Z_h$ is preserved by left and right actions, and by definition we
have
\[
\raisebox{-1pt}{$T^n$}\backslash\raisebox{1pt}{$Z_h$}=M_h\cap
M_\Lambda=X_h.
\]
The condition $U\in Z_h$ implies $U^{-1}\Lambda U (E_i)\subset
E_{h(i)}$, or equivalently $\Lambda U(E_i)\subset U(E_{h(i)})$.
Hence $\Lambda(F^U_i)\subset F^U_{h(i)}$ and we have
\[
\raisebox{1pt}{$Z_h$}/\raisebox{-1pt}{$T^n$}=\{F^U_\bullet\mid
U\in Z_h\}=Y_h.
\]
It can be seen that both orbit spaces $X_h/T^n$ and $Y_h/T^n$
coincide with the double quotient
\raisebox{-1pt}{$T^n$}$\backslash$\raisebox{1pt}{$Z_h$}$/$\raisebox{-1pt}{$T^n$}.
Filtrations by orbit dimensions coincide with the corresponding
filtration on the space
\raisebox{-1pt}{$T^n$}$\backslash$\raisebox{1pt}{$Z_h$}$/$\raisebox{-1pt}{$T^n$},
which proves the first statement.

In fact we used the following idea. We consider a map from the set
of Hermitian matrices with fixed spectrum to itself, given by the
rule
\[
L=U\Lambda U^{-1}\mapsto U^{-1}\Lambda U=\tilde{L}.
\]
In a sense, this rule maps $X_h$ to $Y_h$ and vice versa. However
this map is ill-defined: on one side $U$ is defined up to right
action of torus, and on the other side it is defined up to left
action. The ambiguity vanishes after passing to double quotient.

(2) The unlabeled GKM-graph of $X_h$ is described earlier in this
section. Its vertices are all permutations $\sigma\in E$, and two
vertices $\sigma$ and $\tau$ are incident if $\tau=\sigma(i,j)$
and either $i<j\leqslant h(i)$ or $j<i\leqslant h(j)$. The
unlabeled GKM-graph of $Y_h$ is the same, see \cite{AHM}.

(3) We recall the result of \cite{AHM} on the structure of
equivariant cohomology ring of a semisimple Hessenberg variety:
\[
H^*_T(X_h;\Zo)\cong\left\{\phi\colon \Sigma_n\to
H^*(BT^n;\Zo)\left|
\pbox{\textwidth}{$\phi(\sigma)\equiv\phi(\tau)\mod
(\epsilon_{\sigma(i)}-\epsilon_{\sigma(j)})$\\ если
$\tau=\sigma\cdot(i,j)$ и $i<j\leqslant h(i)$  либо $j<i\leqslant
h(j)$ }\right.\right\}
\]
This differs from the case of $X_h$ (proposition
\ref{propEquivCohomXh}) only in relation
$\epsilon_{\sigma(i)}-\epsilon_{\sigma(j)}$ instead of relation
$\epsilon_i-\epsilon_j$. The isomorphism of rings can be
constructed explicitly. Let
\[
R=\{\Sigma_n\to H^*(BT)\}=\bigoplus_{\sigma\in\Sigma_n}
H^*(BT)_\sigma,
\]
where $H^*(BT)_\sigma$ is the copy of the ring
$H^*(BT;\Zo)\cong\Zo[\epsilon_1,\ldots,\epsilon_n]$,
$\deg\epsilon_i=2$. Define the homomorphism $\twist\colon R\to R$,
which acts on the summand $H^*(BT)_\sigma$ by switching the
generators according to permutation $\sigma$:
\[
\twist|_{H^*(BT)_\sigma}\colon
\Zo[\epsilon_1,\ldots,\epsilon_n]\to
\Zo[\epsilon_1,\ldots,\epsilon_n],\qquad \epsilon_i\mapsto
\epsilon_{\sigma(i)}.
\]
It can be seen that $\twist$ induces the isomorphism of the
subring $H^*_T(X_h;\Zo)\subset R$ onto the subring
$H^*_T(Y_h;\Zo)\subset R$.

Note that there is a simpler implicit proof which uses the
construction of (1). Indeed, both manifolds $X_h$ and $Y_h$ are
the quotient of the same manifold $Z_h$ by free action of a torus.
Hence both rings $H^*_T(X_h;\Zo)$, $H^*_T(Y_h;\Zo)$ are isomorphic
to $H^*_{T\times T}(Z_h;\Zo)$.

(4) For an equivariantly formal space $X$ the equivariant
cohomology is a free module over $H^*(BT)$, and there is an
isomorphism $H^*(X)\cong H^*_T(X)\otimes_{H^*(BT)}\Zo$. It follows
that $\sum_i\rk H^{i}_T(X)t^{i}=\dfrac{\sum_i\rk
H^{i}(X)t^{i}}{(1-t^2)^n}$. Since both spaces $X_h$, $Y_h$ are
equivariantly formal, the coincidence of their Betti numbers
follows from the isomorphism of equivariant cohomology rings:
\[
\dfrac{\sum_k\rk H^{2k}(X_h)t^{2k}}{(1-t^2)^n}=\sum_k\rk
H^{2k}_T(X_h)t^{2k}=\sum_k\rk H^{2k}_T(Y_h)t^{2k}=\dfrac{\sum_k\rk
H^{2k}(Y_h)t^{2k}}{(1-t^2)^n}.
\]
Also note that Betti number $\beta_{2k}(X_h)=\rk H^{2k}(X_h;\Zo)$
equals the number of equilibria points of index $2k$, that is
\[
\beta_{2k}(X_h)=\sharp\{\sigma\in \Sigma_n\mid
\sharp\{i<j\leqslant h(i)\mid \sigma(j)<\sigma(i)\}=k\}.
\]
\end{proof}

\begin{rem}
Floyd theorem was used in \cite{dMP} to attack the real case. It
was proved that
\[
\dim H^k(X_h^\Ro;\Zt)=\beta_{2k}(X_h).
\]
One can also define the real analog of Hessenberg manifold. Let
$Y_h^\Ro$ be the manifold consisting of all real flags $V_\bullet$
such that $\Lambda V_i\subset V_{h(i)}$. This manifold also
carries the action of the group $\Zt^n$ and Floyd theorem implies
a similar result:
\[
\dim H^k(X_h^\Ro;\Zt)\cong \dim H^k(Y_h^\Ro;\Zt).
\]

It should be noted that manifolds $X_h^\Ro$ and $Y_h^\Ro$ are
generally not diffeomorphic, which is similar to example
\ref{exXYareDifferent}. For $n=3$ and $h=\hmin$ the manifold
$X_h^\Ro$ is a sphere with two handles, however $Y_h^\Ro\cong
4\RP^2$ is a nonorientable surface.
\end{rem}

\section{Cohomology rings}

Notice that for the Hessenberg function $\hmin$, the manifold
$X_{\hmin}$ is a quasitoric manifold. The cohomology ring and
equivariant cohomology ring of a quasitoric manifold can be
written in terms of generators and relations, see \cite{DJ}. In
particular, it is known that both rings are generated by elements
of degree two.

Also note that the maximal possible Hessenberg function $\hm(i)=n$
gives $X_{\hm}=\Fl_n$. The cohomology ring of a complete flag
manifold is also generated in degree two. Is it true that
$H^*(X_h)$ is generated by $H^2(X_h)$ in general? The answer is
negative.

\begin{prop}
There exists Hessenberg functions $h$ such that $H^*(X_h)$ has
generators of degrees $>2$.
\end{prop}

\begin{proof}
For equivariantly formal space $X$, the ring $H^*(X)$ is a
quotient of $H^*_T(X)$ by the ideal generated in degree $2$.
Therefore, $H^*(X)$ is generated in degree $2$ whenever $H^*_T(X)$
is generated in degree $2$. Now, since $H^*_T(X_h)$ is isomorphic
to $H^*_T(Y_h)$ as ring, we see that $H^*(X_h)$ is generated in
degree $2$ if and only if $H^*(Y_h)$ is generated in degree $2$.
It was shown in \cite{AHM} that the Hessenberg function
$h=(h(1),n,\ldots,n)$ produces the Hessenberg variety $Y_h$ whose
cohomology ring has generators of degree $2(h(1)-1)$. Therefore,
the cohomology ring of the corresponding manifold $X_h$ is not
generated in degree two as well.
\end{proof}

\end{document}